\documentclass[11pt,a4paper]{article}

\usepackage[utf8]{inputenc}
\usepackage[T1]{fontenc}
\usepackage{amsmath,amssymb,amsthm}
\usepackage{geometry}
\usepackage{hyperref}

\newtheorem{theorem}{Theorem}[section]
\newtheorem{lemma}[theorem]{Lemma}
\newtheorem{remark}{Remark}[section]

\newcommand{\wt}{\widetilde}
\newcommand{\floor}[1]{\lfloor #1 \rfloor}
\newcommand{\ceil}[1]{\lceil #1 \rceil}
\DeclareMathOperator{\E}{\mathbb{E}}

\begin{document}

\title{Uniform transport processes and the Brownian sheet:\\
       the one-parameter rate of convergence is preserved}

\author{Carles Rovira%
  \thanks{Facultat de Matem\`atiques i Inform\`atica,
    Universitat de Barcelona, Gran Via 585, 08007 Barcelona, Spain.
    E-mail: \texttt{Carles.Rovira@ub.edu}.
    Supported by grant PID2025-173091NB-I00
    funded by MCIN/AEI/10.13039/501100011033.}}

\date{ October 16, 2026}

\maketitle



\begin{abstract}
In~\cite{BFR} we constructed processes $W_n$, built from a family of
independent standard Poisson processes, that converge almost surely to the
Brownian sheet uniformly on $[0,1]^2$, and in~\cite{Rov} we proved that
$P(\max_{[0,1]^2}|W_n-W|>\alpha n^{-\beta})=o(n^{-q})$ for every
$\beta<\lambda/2$, where $\lambda\in(0,\tfrac{1}{5})$ is the parameter of the
construction; this gives rates $n^{-\beta}$ with $\beta<\tfrac{1}{10}$.
In this paper, we show that the same construction, with $\lambda\geq 1$,
satisfies
\[
  P\!\left(\max_{(s,t)\in[0,1]^2}
     |W_n(s,t)-W(s,t)| > \alpha\, n^{-1/2}(\log n)^3\right) = o(n^{-q})
  \quad\text{for every }q>0,
\]
which is the one-parameter rate of Gorostiza and Griego~\cite{GG} up to a
factor $(\log n)^{1/2}$.  In particular the restriction $\lambda<\tfrac{1}{5}$
is not needed.
\end{abstract}

\section{Introduction}

Let $W=\{W(s,t):(s,t)\in[0,1]^2\}$ be a Brownian sheet.
Griego, Heath and Ruiz-Moncayo~\cite{GHR} constructed realizations of uniform
transport processes converging almost surely to Brownian motion, uniformly on
$[0,1]$; the rate of convergence was obtained by Gorostiza and
Griego~\cite{GG} and by Cs\"{o}rg\H{o} and Horv\'{a}th~\cite{CH}.
In \cite{BFR} we extended the almost sure convergence
to the Brownian sheet: they cut $[0,1]^2$ into $\floor{n^\lambda}$ vertical
strips, observe that the sheet integrated over each strip is, after rescaling,
a standard Brownian motion, approximate each of these independent Brownian
motions by a uniform transport process as in~\cite{GHR}, and sum over the
strips. They obtained almost sure uniform convergence for
$\lambda\in(0,\tfrac{1}{5})$ and left the rate of convergence as an open
problem. In ~\cite{Rov}  we obtained the rate
\[
  P\!\left(\max_{[0,1]^2}|W_n-W|>\alpha n^{-\beta}\right)=o(n^{-q}),
  \qquad\beta<\frac{\lambda}{2},\quad
  \lambda\in\!\left(0,\tfrac{1}{5}\right),
\]
and observed~\cite[Remark~2.3]{Rov} that it is worse than the one-parameter
rate $n^{-1/2}(\log n)^{5/2}$ of~\cite{GG}, because the latter, valid for
each of the $\floor{n^\lambda}$ Brownian motions, could not be kept for their
sum.

The purpose of this note is to show that the one-parameter rate is kept for
the sum, up to a factor $(\log n)^{1/2}$.  In~\cite{BFR,Rov} the errors of
the different strips are added through the triangle inequality,  but these errors are
independent and centered, so their partial sums behave like a random walk
rather than like a sum of absolute values.  On the other hand, using the
Borell--TIS inequality, it no longer constrains $\lambda$ from above, so we
can take $\lambda\geq 1$.

Section~\ref{sec:construction} recalls the construction and states precisely,
as a hypothesis, the properties of the one-parameter approximation that are
used.  Section~\ref{sec:main} contains the results and Section~\ref{sec:proof}
the proof.  The probabilistic tools used in the proof are collected in
Section~\ref{sec:tools}.

\section{The construction and the one-parameter input}
\label{sec:construction}

We follow~\cite{BFR,Rov}.  Fix $\lambda>0$ and write $L=\floor{n^\lambda}$,
assuming $n$ large enough that $L\geq 1$.  It is convenient to assume that
$W$ is the restriction to $[0,1]^2$ of a Brownian sheet defined on
$[0,2]\times[0,1]$; this is no loss of generality, since the statements are
about the existence of suitable realizations.  For $k=1,\ldots,L$ put
\[
  W_k(t)=W\!\left(\tfrac{k}{n^\lambda},t\right)
          -W\!\left(\tfrac{k-1}{n^\lambda},t\right),
  \qquad
  \wt W_k(t)=n^{\lambda/2}W_k(t),\quad t\in[0,1].
\]
The processes $\wt W_1,\ldots,\wt W_L$ are independent standard Brownian
motions on $[0,1]$ (they are centered, Gaussian, continuous, and
$\E[\wt W_k(t)\wt W_{k'}(t')]=(t\wedge t')\mathbf{1}_{\{k=k'\}}$), and
$W(l/n^\lambda,t)=\sum_{k\leq l}W_k(t)$.

For each $k$, a uniform transport process $\wt W_k^{(n)}$ is constructed
from $\wt W_k$ by the Skorohod embedding procedure of~\cite{GHR}.
Throughout the paper we write
\[
  r_n=n^{-1/2}(\log n)^{5/2},\qquad n\geq 2,
\]
for the one-parameter rate of Gorostiza and Griego~\cite{GG}.
In the proofs we will use that
\begin{equation}\label{eq:rn-bound}
  r_n\geq n^{-1}\quad\text{for all }n\geq 3.
\end{equation}
We use the normalization of~\cite{GG}, in which the transport process moves
with speed $n$.  The transport process of the $k$-th strip has the law of
\[
  X_n(t)=n(-1)^A\int_0^t(-1)^{N(n^2 r)}\,dr
         =\frac{1}{n}(-1)^A\int_0^{n^2 t}(-1)^{N(u)}\,du,
  \qquad t\in[0,1],
\]
where $N$ is a standard Poisson process and $A$ an independent
Bernoulli$(\tfrac{1}{2})$ variable: the process moves at speed $n$ and
reverses its direction at the jump times of a Poisson process of rate $n^2$.
It is realized on the same probability space as $\wt W_k$ as
in~\cite{GHR,BFR}.  This is also the process constructed in the proof
of~\cite[Theorem~2.1]{BFR}: there the pieces have slope $\pm n$, durations
exponential with parameter $2n^2$ and i.i.d.\ symmetric signs, and since each
new sign differs from the previous one with probability $\tfrac{1}{2}$,
independently, the direction is reversed at the jump times of a Poisson
process of rate $n^2$.

From the construction of the approximations we get the following properties:

\begin{itemize}
\item[\textbf{(H1)}] For every $n$ the pairs $(\wt W_k,\wt W_k^{(n)})$,
  $k\leq L$, are independent and identically distributed.  In particular
  the error processes $D_k=\wt W_k^{(n)}-\wt W_k$, $k\leq L$, are
  independent and identically distributed.

\item[\textbf{(H2)}] For every $n$, $k$ and $t\in[0,1]$,
  $\E\wt W_k^{(n)}(t)=0$; and almost surely, for all $t,t'\in[0,1]$,
  \[
    |\wt W_k^{(n)}(t)|\leq n,\qquad
    |\wt W_k^{(n)}(t)-\wt W_k^{(n)}(t')|\leq n|t-t'|.
  \]

\item[\textbf{(H3)}] For every $q>0$ there are $\alpha_q>0$ and $n_q\geq 3$,
  depending on $q$ but not on $n$ or $k$, such that, with
  $r_n=n^{-1/2}(\log n)^{5/2}$,
  \begin{equation}\label{eq:H3}
    P\!\left(\max_{0\leq t\leq 1}
      |\wt W_k^{(n)}(t)-\wt W_k(t)|>\alpha_q r_n\right)\leq n^{-q}
    \quad\text{for all }n\geq n_q\text{ and all }k\leq L.
  \end{equation}
\end{itemize}

The approximating process is defined at the strip boundaries by
\[
  W_n\!\left(\tfrac{l}{n^\lambda},t\right)=\sum_{k=1}^l W_k^{(n)}(t),
  \qquad
  W_k^{(n)}(t)=n^{-\lambda/2}\wt W_k^{(n)}(t),\quad 0\leq l\leq L,
\]
by linear interpolation in $s$ on $[0,L/n^\lambda]$, and by
$W_n(s,t)=W_n(L/n^\lambda,t)$ for $s\in[L/n^\lambda,1]$.

\section{Main theorem}
\label{sec:main}

The new bounds on the errors for these approximations are collected in the
following theorem.

\begin{theorem}\label{thm:main-bound}
Let $\lambda>0$.  For every $q>0$ there are constants $C$ and $n_0$,
depending on $q$, $\lambda$ and on the constants of~\emph{(H3)}, such that
for all $n\geq n_0$
\[
  P\!\left(\max_{(s,t)\in[0,1]^2}|W_n(s,t)-W(s,t)|
    >C\sqrt{\log n}\!\left(\frac{(\log n)^{5/2}}{\sqrt{n}}
      +n^{-\lambda/2}\right)\right)\leq n^{-q}.
\]
\end{theorem}

Finally, the main result of this paper is collected in the following theorem.

\begin{theorem}\label{thm:main-rate}
For every $q>0$ there are $\alpha>0$ and $n_0$ such that for $n\geq n_0$
\[
  P\!\left(\max_{(s,t)\in[0,1]^2}|W_n(s,t)-W(s,t)|
    >\alpha\, n^{-1/2}(\log n)^3\right)\leq n^{-q},
\]
and $\max_{[0,1]^2}|W_n-W|=O(n^{-1/2}(\log n)^3)$ almost surely.
In particular the almost sure convergence of $W_n$ to $W$ holds.
\end{theorem}

\begin{proof}
Recall $r_n=n^{-1/2}(\log n)^{5/2}$ and apply Theorem~\ref{thm:main-bound}.
Choose any $\lambda\geq 1$.  Then $n^{-\lambda/2}\leq n^{-1/2}\leq r_n$, so
the bound is at most $2Cr_n(\log n)^{1/2}$.

Using $q=2$: with
$A_n=\{\max_{[0,1]^2}|W_n-W|>\alpha n^{-1/2}(\log n)^3\}$ we have
\[
  \sum_{n\geq 1}P(A_n)\leq n_0+\sum_{n\geq n_0}n^{-2}<\infty.
\]
By the Borel--Cantelli lemma, almost surely $A_n$ occurs for finitely many
$n$ only, i.e.\ $\max_{[0,1]^2}|W_n-W|\leq\alpha n^{-1/2}(\log n)^3$ for
all $n$ large enough.
\end{proof}

\begin{remark}\label{rem:rate-comparison}
The rates of~\cite{Rov} are $n^{-\beta}$ with $\beta<\lambda/2<\tfrac{1}{10}$.
Theorem~\ref{thm:main-rate} gives $n^{-1/2}(\log n)^3$, the rate
$n^{-1/2}(\log n)^{5/2}$ of~\cite{GG} for one Brownian motion up to a factor
$(\log n)^{1/2}$.  This answers the question raised in~\cite[Remark~2.3]{Rov}.
\end{remark}

\begin{remark}\label{rem:d-param}
\emph{($d$-parameter Wiener processes.)}
The same method applies to the $d$-parameter Wiener process, $d\geq 2$,
approximated as in~\cite[Section~4]{BFR}: the strips are replaced by the
$\floor{n^\lambda}^{d-1}$ boxes
$\prod_{i<d}((k_i-1)/n^\lambda,k_i/n^\lambda]\times[0,1]$, each carrying an
independent transport process, the normalization $n^{-\lambda/2}$ becomes
$n^{-\lambda(d-1)/2}$, and $W_n$ is defined by $(d-1)$-linear interpolation.
\end{remark}

\section{Proof of Theorem~\ref{thm:main-bound}}
\label{sec:proof}

Fix $q>0$. All constants below may depend on $q$, $\lambda$ and the constants
of (H3), but not on $n$, and $n$ is assumed large enough.
\subsection{Reduction of the problem}

For $0\leq l\leq L$ and $t\in[0,1]$ put
\[
  \Delta_n(l,t)
  =W_n\!\left(\tfrac{l}{n^\lambda},t\right)
   -W\!\left(\tfrac{l}{n^\lambda},t\right)
  =n^{-\lambda/2}\sum_{k=1}^l D_k(t),
  \qquad D_k=\wt W_k^{(n)}-\wt W_k,
\]
and, for $0\leq l\leq L$,
\[
  Z_l=\sup\!\left\{
    \left|W(s,t)-W\!\left(\tfrac{l}{n^\lambda},t\right)\right|:
    s\in\left[\tfrac{l}{n^\lambda},\tfrac{l+1}{n^\lambda}\right],\,
    t\in[0,1]\right\}.
\]
The following lemma allows us to reduce the study of $|W_n(s,t)-W(s,t)|$ to
the study of $\max_{0\leq l\leq L}|\Delta_n(l,t)|$ and
$\max_{0\leq l\leq L}Z_l$.

\begin{lemma}\label{lem:det-reduction}
For every $(s,t)\in[0,1]^2$,
\[
  |W_n(s,t)-W(s,t)|\leq\max_{0\leq l\leq L}|\Delta_n(l,t)|
                        +2\max_{0\leq l\leq L}Z_l.
\]
\end{lemma}

\begin{proof}
If $s\in[l/n^\lambda,(l+1)/n^\lambda)$ with $l\leq L-1$, write
$\theta=n^\lambda s-l\in[0,1)$.  Then
\begin{align*}
  W_n(s,t)-W(s,t)
  &=(1-\theta)\Delta_n(l,t)+\theta\Delta_n(l+1,t)\\
  &\quad+(1-\theta)\!\left[W\!\left(\tfrac{l}{n^\lambda},t\right)-W(s,t)\right]
       +\theta\!\left[W\!\left(\tfrac{l+1}{n^\lambda},t\right)-W(s,t)\right].
\end{align*}
The first bracket is bounded by $Z_l$ and the second, by the triangle
inequality through $W(l/n^\lambda,t)$, by $2Z_l$.  If $s\geq L/n^\lambda$,
then $W_n(s,t)-W(s,t)=\Delta_n(L,t)+W(L/n^\lambda,t)-W(s,t)$, and the last
difference is bounded by $Z_L$.
\end{proof}

Thus it suffices to prove that for $n$ large enough,
\begin{align}
  P\!\left(\max_{0\leq l\leq L}Z_l
    >Cn^{-\lambda/2}(\log n)^{1/2}\right)&\leq n^{-q},
  \label{eq:bound-Z}\\
  P\!\left(\max_{0\leq l\leq L}\sup_{t\in[0,1]}|\Delta_n(l,t)|
    >C\,r_n(\log n)^{1/2}\right)&\leq n^{-q}.
  \label{eq:bound-Delta}
\end{align}

\subsection{Study of~\eqref{eq:bound-Z}}

For each $l$, the field
\[
  B_l(\rho,t)=n^{\lambda/2}\,W\!\left(\left(\tfrac{l}{n^\lambda},
    \tfrac{l}{n^\lambda}+\tfrac{\rho}{n^\lambda}\right]\times(0,t]\right),
  \qquad(\rho,t)\in[0,1]^2,
\]
is a Brownian sheet, and $Z_l=n^{-\lambda/2}\sup_{[0,1]^2}|B_l|$.
Let $m=\E\sup_{[0,1]^2}|B_l|$, which does not depend on $l$.
By Lemma~\ref{lem:borell} (maximal variance~$1$) with
$x=\sqrt{2(q+\lambda+1)\log n}$ and a union bound over the
$L+1\leq n^\lambda+1$ strips,
\[
  P\!\left(\max_l Z_l>n^{-\lambda/2}\!\left(m+\sqrt{2(q+\lambda+1)\log n}
    \right)\right)\leq 2(n^\lambda+1)n^{-(q+\lambda+1)}\leq 4n^{-q-1}.
\]
We get the bound~\eqref{eq:bound-Z}.

\subsection{Study of~\eqref{eq:bound-Delta}}

We begin studying $D_k$.  Let $q'=q+\lambda+7$, $a_n=\alpha_{q'}r_n$, with
$\alpha_{q'}$ given by (H3), and put
\[
  F_k=\left\{\sup_{t\in[0,1]}|D_k(t)|\leq a_n\right\},
  \qquad G_n=\bigcap_{k=1}^L F_k.
\]
By (H3) and (H1), $\pi_n:=P(F_k^c)$ does not depend on $k$ and
$\pi_n\leq n^{-q'}$ for $n\geq n_{q'}$; hence
\[
  P(G_n^c)\leq L\pi_n\leq n^{\lambda-q'}\leq n^{-q-1}.
\]
We apply Lemma~\ref{lem:truncation} to each $D_k$.  By (H2), $\E D_k(t)=0$
and $\sup_t|D_k(t)|\leq Y_k:=n+\sup_t|\wt W_k(t)|$, so that
$\E Y_k^2\leq 2n^2+8\leq 10n^2$ by Doob's inequality.  By (H1) the function
$m_n(t)=\E[D_k(t)\mathbf{1}_{F_k}]$ does not depend on $k$, and
\begin{equation}\label{eq:mn-bound}
  |m_n(t)|\leq(10n^2)^{1/2}\pi_n^{1/2}
  \leq\sqrt{10}\,n^{1-q'/2}=\sqrt{10}\,n^{-(q+\lambda+5)/2}.
\end{equation}
On $G_n$ we have $D_k(t)=D_k(t)\mathbf{1}_{F_k}=D'_k(t)+m_n(t)$ for every
$k$ and $t$, hence
\begin{equation}\label{eq:Delta-decomp}
  \Delta_n(l,t)=n^{-\lambda/2}\sum_{k=1}^l D'_k(t)+n^{-\lambda/2}l\,m_n(t),
  \qquad
  |n^{-\lambda/2}l\,m_n(t)|\leq n^{\lambda/2}|m_n(t)|\leq n^{-3/2},
\end{equation}
using $l\leq n^\lambda$ and~\eqref{eq:mn-bound}.  Moreover, for fixed $t$,
the variables $D'_1(t),\ldots,D'_L(t)$ are independent (by (H1)) and centered,
and by Lemma~\ref{lem:truncation}
\begin{equation}\label{eq:Dk-bound}
  |D'_k(t)|\leq a_n+|m_n(t)|\leq 2a_n
  \quad\text{surely, for every }k\leq L\text{ and }t\in[0,1],
\end{equation}
because $|m_n(t)|\leq\sqrt{10}\,n^{-5/2}\leq\alpha_{q'}n^{-1}\leq a_n$ for
$n$ large, by~\eqref{eq:mn-bound} and $r_n\geq n^{-1}$.

\subsubsection{Partial sums on a grid}

Let $M=\ceil{n^{\lambda+3}}$, $t_i=i/M$, $i=0,\ldots,M$.
For fixed $l\leq L$ and $i\leq M$, the variables $X_k=D'_k(t_i)$, $k\leq l$,
are independent, centered and, by~\eqref{eq:Dk-bound}, bounded by the
deterministic constant $c=2a_n$, uniformly in $k$, $l$ and $i$.
Lemma~\ref{lem:hoeffding} therefore applies with $c=2a_n$ and gives, since
$l\leq n^\lambda$,
\[
  P\!\left(n^{-\lambda/2}\left|\sum_{k=1}^l D'_k(t_i)\right|>y\right)
  \leq 2\exp\!\left(-\frac{y^2 n^\lambda}{2l(2a_n)^2}\right)
  \leq 2\exp\!\left(-\frac{y^2}{8a_n^2}\right).
\]
Choose $y_n=a_n\sqrt{8(q+2\lambda+4)\log n}$.  The number of pairs $(l,i)$ is
at most $(n^\lambda+1)(M+1)\leq 4n^{2\lambda+3}$, so a union bound yields
\[
  P\!\left(\max_{l\leq L}\max_{i\leq M}
    n^{-\lambda/2}\left|\sum_{k=1}^l D'_k(t_i)\right|>y_n\right)
  \leq 8n^{2\lambda+3}n^{-(q+2\lambda+4)}=8n^{-q-1}.
\]
Together with~\eqref{eq:Delta-decomp} and the bound on $P(G_n^c)$, outside
an event of probability at most $9n^{-q-1}$,
\begin{equation}\label{eq:grid-bound}
  \max_{l\leq L}\max_{i\leq M}|\Delta_n(l,t_i)|
  \leq y_n+n^{-3/2}\leq C\,r_n(\log n)^{1/2},
\end{equation}
where we used $n^{-3/2}\leq r_n$.

\subsubsection{Between grid points}

Let $t\in[t_i,t_{i+1}]$.  Then
\[
  |\Delta_n(l,t)-\Delta_n(l,t_i)|
  \leq\left|W_n\!\left(\tfrac{l}{n^\lambda},t\right)
             -W_n\!\left(\tfrac{l}{n^\lambda},t_i\right)\right|
      +\left|W\!\left(\tfrac{l}{n^\lambda},t\right)
             -W\!\left(\tfrac{l}{n^\lambda},t_i\right)\right|.
\]
By (H2), the first term is at most
$n^{-\lambda/2}\cdot L\cdot n\cdot M^{-1}\leq n^{\lambda/2+1}n^{-(\lambda+3)}
 \leq n^{-2}$, deterministically.
For the second, for each $i<M$ the field
$(s,\rho)\mapsto M^{1/2}W((0,s]\times(t_i,t_i+\rho/M])$ on $[0,1]^2$ is a
Brownian sheet, and the second term is bounded by $M^{-1/2}$ times its
supremum.  By Lemma~\ref{lem:borell} with
$x=\sqrt{2(q+\lambda+4)\log n}$ and a union bound over the
$M\leq 2n^{\lambda+3}$ intervals,
\[
  P\!\left(\max_{i<M}\sup_{s\in[0,1],\,t\in[t_i,t_{i+1}]}
    |W(s,t)-W(s,t_i)|>M^{-1/2}\!\left(m+\sqrt{2(q+\lambda+4)\log n}
    \right)\right)\leq 4n^{-q-1}.
\]
Since $M^{-1/2}\leq n^{-1}\leq r_n$, combining with~\eqref{eq:grid-bound}
we obtain the bound~\eqref{eq:bound-Delta}.  \qed

\begin{remark}\label{rem:rate-lost}

The construction of Section~\ref{sec:construction} makes sense for every
$\lambda>0$; nothing in the definition of the strips, of the strip Brownian
motions $\wt W_k$, of the transport processes $\wt W_k^{(n)}$ or of the
interpolated process $W_n$ requires $\lambda<\tfrac{1}{5}$.  The restriction
$\lambda\in(0,\tfrac{1}{5})$ appears only in the proof of almost sure
convergence in~\cite[pp.~324--325]{BFR}, where Markov's inequality of order
$M$ leads to an exponent that can be made smaller than $-1$ only if
$\lambda<\tfrac{1}{5}$. 

The present paper uses from~\cite{BFR} only the construction
(hypotheses (H1)--(H2)) and from~\cite{GG} only the one-parameter rate of
each strip (hypothesis (H3)); none of these depends on $\lambda$.  Therefore
Theorem~\ref{thm:main-bound} and Theorem~\ref{thm:main-rate} hold for every
$\lambda>0$, and in particular for $\lambda\geq 1$, which is the range giving
the rate $n^{-1/2}(\log n)^3$.  
\end{remark}

\section{Probabilistic tools}
\label{sec:tools}

We collect here the probabilistic inequalities used in
Section~\ref{sec:proof}.  We begin with a Borell--TIS inequality.

\begin{lemma}\label{lem:borell}
Let $\Gamma$ be a centered Gaussian field with continuous paths on a compact
set $T$, let $\sigma^2=\sup_{x\in T}\E\Gamma(x)^2$ and
$m=\E\sup_T|\Gamma|$.  Then $m<\infty$ and, for every $x>0$,
\[
  P\!\left(\sup_T|\Gamma|>m+x\right)\leq 2e^{-x^2/(2\sigma^2)}.
\]
\end{lemma}

\begin{proof}
Since $\Gamma$ has continuous paths on the compact set $T$, it is a.s.\
bounded: for almost every $\omega$ the map $t\mapsto\Gamma(t,\omega)$ is
continuous on a compact set, hence bounded. The Borell--TIS
inequality~\cite[Theorem~2.1.1]{AT}, applied to $\Gamma$ and to $-\Gamma$,
then gives
\[
  P\!\left(\sup_T \Gamma > \E\sup_T \Gamma + x\right)\leq e^{-x^2/(2\sigma^2)},
  \qquad
  P\!\left(\inf_T \Gamma < \E\inf_T \Gamma - x\right)\leq e^{-x^2/(2\sigma^2)}.
\]
Since $\E\sup_T(-\Gamma)=-\E\inf_T\Gamma\leq\E\sup_T|\Gamma|=m$, combining
the two inequalities yields the stated bound.
\end{proof}

We will also use the following Hoeffding inequality (see~\cite{Hoe}).

\begin{lemma}\label{lem:hoeffding}
If $X_1,\ldots,X_l$ are independent, centered and $|X_k|\leq c$ almost
surely for every $k$, then for every $x>0$
\[
  P\!\left(\left|\sum_{k=1}^l X_k\right|>x\right)
  \leq 2\exp\!\left(-\frac{x^2}{2lc^2}\right).
\]
\end{lemma}

The last technical result is a truncation lemma.

\begin{lemma}\label{lem:truncation}
Let $D=(D(t))_{t\in[0,1]}$ be a process with $\E D(t)=0$ for every $t$ and
$\sup_t|D(t)|\leq Y$ with $\E Y^2<\infty$.  For $a>0$ let
$F=\{\sup_t|D(t)|\leq a\}$ and
\[
  m(t)=\E\!\left[D(t)\mathbf{1}_F\right],\qquad
  D'(t)=D(t)\mathbf{1}_F-m(t).
\]
Then $\E D'(t)=0$, $|m(t)|\leq(\E Y^2)^{1/2}P(F^c)^{1/2}$ and, surely,
$|D'(t)|\leq a+|m(t)|$ for every $t$.
\end{lemma}

\begin{proof}
Since $\E D(t)=0$, we have $m(t)=-\E[D(t)\mathbf{1}_{F^c}]$, and
Cauchy--Schwarz gives the bound on $m(t)$.  The bound on $D'(t)$ holds on
every sample point, because $|D(t)\mathbf{1}_F|\leq a$ everywhere.
\end{proof}

\section{Declaration of generative AI and AI-assisted technologies in the manuscript preparation process.}
During the preparation of this work the author used Claude (Anthropic) to assist with LaTeX typesetting of formulas, translation and language editing, and the verification of proofs and calculations. After using this tool, the author reviewed and edited the content as needed and takes full responsibility for the content of the published article.


\end{document}